\documentclass[a4paper,reqno,12pt]{amsart}

\usepackage[T1]{fontenc}
\usepackage{amsthm} % you can keep amsthm
\usepackage{amsmath, amssymb, epsfig}
\usepackage{url}
\usepackage[mathscr]{euscript}
\usepackage{color}
\usepackage{harpoon}
\usepackage{tikz}
\usepackage{pgfplots}
\pgfplotsset{compat=1.18}
\usepackage{float}
\usetikzlibrary{shapes.geometric, calc}
\usepackage{newtxtext}
\usepackage{newtxmath}
\usepackage{hyperref}

\usepackage{fullpage} 
\usepackage{setspace}
\usepackage{adjustbox}

\usepackage{mathtools}
\mathtoolsset{showonlyrefs}
\def\today{\ifcase\month\or
  January\or February\or March\or April\or May\or June\or
  July\or August\or September\or October\or November\or December\fi
  \space\number\day, \number\year}

\newtheorem{theorem}{Theorem}
\newtheorem{conjecture}{Conjecture}

\newtheorem{corollary}[theorem]{Corollary}

\newcommand{\E}{\mathbb{E}}

\newcommand{\Z}{\mathscr{Z}}

\newcommand{\z}{\mathbb{Z}}

\renewcommand{\r}{\mathbb{R}}
\newcommand{\be}{\beta}

\newcommand{\si}{\sigma}

\newcommand{\p}{\varphi}

\newcommand{\rp}{\r_{\geq 0}}

\newcommand{\maj}{\prec_{\text{maj}}}
\newcommand{\conv}{\prec_{\text{conv}}}

\begin{document}

%------------------HEADINGS------------------------

\title[]{Convex order, log-convexity and Moments of Averages of Random Variables}
\author[Gon\c{c}alves]{Felipe Gon\c{c}alves, Jos\'e Madrid and Antonio Pedro Ramos}
\date{\today}
\address{IMPA - Instituto de Matemática Pura e Aplicada, Rio de Janeiro, 22460-320, Brazil.}
\email{goncalves@impa.br}

\address{Department of  Mathematics, Virginia Polytechnic Institute and State University,  225 Stanger Street, Blacksburg, VA 24061-1026, USA}
\email{josemadrid@vt.edu}

\address{IMPA - Instituto de Matemática Pura e Aplicada, Rio de Janeiro, 22460-320, Brazil.}
\email{antonio.ramos@impa.br}

\allowdisplaybreaks
\usetikzlibrary{patterns}
%\numberwithin{theorem}{subsection}{section}
%\numberwithin{equation}{section}

%------------------ABSTRACT------------------------------

\begin{abstract}
We establish a convex order comparison for products of averages of exchangeable nonnegative random variables. As a consequence, we solve a conjecture of Lamkin and Tkocz (2022) on the log-convexity of moments of sample means, as well as its pointwise version, in a more general form involving arbitrary integer
partitions.
\end{abstract}

%---------------------TITLE--------------------------------

\maketitle

%---------------------HAVE--FUN!-----------------------------------

\section{Introduction}
The classical theory of majorization provides stochastic order comparisons
for weighted sums of exchangeable random variables. More precisely, if $X_1,\ldots,X_m$ are exchangeable and $a\maj b$, then
\[
\sum_{i=1}^m a_iX_i\conv\sum_{i=1}^m b_iX_i.
\]
This result goes back to Marshall and Proschan \cite{MP}; see also
\cite{MOA,NPS} for the theory of stochastic majorization and its applications
in probability and statistics. It contains, in particular, the familiar fact
that averages of i.i.d.\ random variables decrease in convex order as the
sample size increases.

The question considered in this paper is nonlinear. We compare products of
normalized sums over disjoint blocks when the vector of block sizes changes
in the majorization order. These products are not covered directly by the
classical comparison for weighted sums, nor by its extension due to Boland,
Proschan and Tong \cite{BPT}, since the product of the block sums does not have
the required joint convexity or concavity.

Our main motivation comes from a conjecture of Lamkin and Tkocz \cite{LT} on
the moments of averages of i.i.d.\ nonnegative random variables. They
conjectured the log-convexity of these moments for powers and
also formulated a stronger pointwise problem. Our main result proves this
pointwise statement and extends it from partitions arising in their
conjecture to arbitrary integer partitions and arbitrary convex test
functions. The corresponding convex order and moment consequences are
explained after the statement of the main result.

\subsection{Main result}

Given two vectors $a,b\in \r^n$ we write $a \maj b$ (and say that $b$ {majorizes} $a$) if 
$$
\sum_{i=1}^j a_i^{\downarrow} \leq  \sum_{i=1}^j b_i^\downarrow
$$
for all $j=1,2,\ldots,n$, with equality when $j=n$, where $\downarrow$ signifies the decreasing rearrangement order. For these and other classical notions we refer to \cite{MOA}. 

The following is the main result of this paper.

\begin{theorem}\label{thm:point}
Let $a,b\in \z_{>0}^n$ be partitions of $m$. Let $A_j = a_1+\dots + a_j$, $B_j=b_1+\dots+b_j$ and $A_0=B_0=0$. Then a necessary and sufficient condition for
\begin{align}\label{ineq:mainpoint}
\sum_{\si\in S_m} \p \bigg[\prod_{j=1}^{n} \bigg(\frac1{a_j}\sum_{i=A_{j-1}+1}^{A_j} X_{\si(i)}\bigg)\bigg] \leq \sum_{\si\in S_m} \p \bigg[\prod_{j=1}^{n} \bigg(\frac1{b_j}\sum_{i=B_{j-1}+1}^{B_j} X_{\si(i)}\bigg)\bigg]
\end{align}
for all real numbers $X_1,X_2,\ldots,X_m\geq 0$ and every convex function $\p:\rp\to\r$ is that $a \maj b$.
\end{theorem}

While the above result pertains to the theory of majorization, it seems to be completely new to the best of our knowledge.  Theorem \ref{thm:point} is proven by constructing an explicit averaging matrix mapping the vectors appearing in inequality \eqref{ineq:mainpoint}. For instance, for the simplest case $(2,2)\maj (3,1)$ we construct
\[
\begin{pmatrix}
{(X_1+X_2)(X_3+X_4)}\\%[2ex]
{(X_1+X_3)(X_2+X_4)}\\%[2ex]
{(X_1+X_4)(X_2+X_3)}
\end{pmatrix}
=
\frac13
\begin{pmatrix}
\lambda_{42}+\lambda_{32} &
\lambda_{41}+\lambda_{31} &
\lambda_{14}+\lambda_{24} &
\lambda_{13}+\lambda_{23}
\\%[1ex]
\lambda_{43}+\lambda_{23} &
\lambda_{14}+\lambda_{34} &
\lambda_{41}+\lambda_{21} &
\lambda_{12}+\lambda_{32}
\\%[1ex]
\lambda_{34}+\lambda_{24} &
\lambda_{13}+\lambda_{43} &
\lambda_{12}+\lambda_{42} &
\lambda_{31}+\lambda_{21}
\end{pmatrix}
\begin{pmatrix}
{X_1(X_2+X_3+X_4)}\\%[2ex]
{X_2(X_1+X_3+X_4)}\\%[2ex]
{X_3(X_1+X_2+X_4)}\\%[2ex]
{X_4(X_1+X_2+X_3)}
\end{pmatrix},
\]
with $\lambda_{ij}:=\frac{X_i}{X_i+X_j}$.

\subsection{Applications}
Let $X,Y$ be two real-valued random variables. We say $X \conv Y$ if for any convex function $\p :\r\to \r$ we have
$$
\E \p(X) \leq \E \p(Y).
$$
A sequence of random variables $X_1,X_2,\dots,X_m$ is exchangeable if \[(X_1,X_2,\dots,X_m) \stackrel{d}{=} (X_{\si(1)},X_{\si(2)},\dots,X_{\si(m)})\] for any permutation $\si \in S_m$. The following is an immediate consequence of Theorem \ref{thm:point}.

\begin{corollary}\label{thm:main}
Let $X_1,\dots,X_m$ be nonnegative exchangeable random variables. Let $a,b\in \z_{> 0}^n$ be partitions of $m$  such that $a \maj b$. Define 
\[
Y_a = \prod_{j=1}^n \left(\frac{1}{a_j}\sum_{i=A_{j-1} +1 }^{A_j} X_i \right) \quad \text{and} \quad Y_b = \prod_{j=1}^n \left(\frac{1}{b_j}\sum_{i=B_{j-1} +1 }^{B_j} X_i \right),
\]
where $A_j = a_1+\dots + a_j$, $B_j=b_1+\dots+b_j$ and $A_0=B_0=0$. Then $Y_a \conv Y_b$.
\end{corollary}

The simplest nontrivial instance of this result is the following inequality
\begin{align}\label{ineq:example}
\E (X_1+X_2)^s(X_3+X_4)^s \leq  (4/3)^s \E (X_1+X_2+X_3)^sX_4^s,
\end{align}
which holds for any $s\in (-\infty,0]\cup [1,\infty)$. The inequality is reversed for $s\in (0,1)$. This particular instantiation can be shown by other classical AM--GM-like or Muirhead-like inequalities in some cases of $s$, but these cannot deal with the general case.

Let $X_1,X_2,\ldots$ be i.i.d.\ nonnegative random variables and write for real $p$
\[
\beta_k(p)=\E \left[\left(\frac{X_1+\dots+X_k}{k}\right)^p \right].
\]
Lamkin and Tkocz have recently made the following conjecture.

\begin{conjecture}[Lamkin--Tkocz \cite{LT}]\label{conj:LT}
For every $p>1$ or $p<0$, and every i.i.d.\ sequence of nonnegative random variables we have
\[
\beta_k(p)^2\leq \beta_{k-1}(p)\beta_{k+1}(p),
\qquad k\geq2.
\]
In other words, $(\be_k(p))_{k\geq 1}$ is log-convex. Moreover, $(\be_k(p))_{k\geq 1}$ is log-concave if $0\leq p\leq 1$.
\end{conjecture}

Lamkin and Tkocz \cite{LT} proved Conjecture \ref{conj:LT} for $p\leq 1$, while for $p>1$ they
had to assume that $k\geq p^2$ and $p$ is an integer. More recently, Jiao and Li \cite{JL} proved it for Gamma random
variables and all $p\geq1$ and for Poisson random variables for integral
$p\geq1$, while Ouimet \cite{Ouimet} proved it for Bernoulli random variables for all $p\geq1$. Lamkin and Tkocz also observed that the theorem of
Boland, Proschan and Tong \cite{BPT} does not resolve the conjecture, since the
function $(x,y)\mapsto (xy)^p$ is neither convex nor concave when $p > 1/2$ and, in the other cases, the obvious application of convexity/concavity yields the inequality with a worse constant.

An immediate consequence of Corollary \ref{thm:main} is the solution of Conjecture \ref{conj:LT} in the following more general form.

\begin{corollary}\label{cor:LTgeneral}
Let $a,b\in\z_{>0}^n$ be partitions of the same integer such that $a\maj b$.
Then for $p \in (-\infty,0]\cup [1,\infty)$ we have
\[
\prod_{j=1}^n\beta_{a_j}(p)
\leq
\prod_{j=1}^n\beta_{b_j}(p).
\]
For $p\in (0,1)$ the inequality is reversed. In particular, the sequence $(\be_k(p))_{k\geq 1}$ is log-convex for $p \in (-\infty,0]\cup[1,\infty)$ and log-concave for $p\in (0,1)$.
\end{corollary}

It is also worth noting that Theorem \ref{thm:point} solves the pointwise conjecture posed by Lamkin and Tkocz in  \cite[Remark 8]{LT}.

\section{Proof of Theorem \ref{thm:point}}
\begin{proof}
\noindent {\bf Sufficiency.} Since $a,b\in \z_{>0}^n$ are partitions of $m$ such that $a\maj b$, it is well-known that $a=T_1 T_2 \cdots T_l b $ for some sequence of doubly stochastic matrices $T_j$, each  with only two off-diagonal terms (the so-called Robin Hood transformations; see \cite{MOA}). In particular, it suffices to show the desired result when $n=2$. We can assume that all $X_i>0$. In this case, for a set $S\subset [m]:=\{1,2,\dots,m\}$ we denote
\[
X_S = \sum_{i\in S} X_i \quad \text{and} \quad Y_S = \frac{X_S X_{[m]\setminus S}}{|S||[m]\setminus S|},
\]
and define the vectors $\Z_k = (Y_S)_{|S|=k}$. Note that $\Z_{k}=\Z_{m-k}$ (upon rearrangement). It is now necessary and sufficient to show that $\Z_{k'}$ is a mean of $\Z_{k}$ for $k<k'\leq m/2$. That is, we have to construct a matrix $[D(R,S)]_{|R|=k',|S|=k}$ with nonnegative entries and the following properties
\begin{align*}
\sum_{|S|=k} D(R,S) = 1, \quad \sum_{|R|=k'} D(R,S) = \frac{\binom{m}{k'}}{\binom{m}k} \quad \text{and} \quad \sum_{|S|=k} D(R,S)Y_S= Y_R.
\end{align*}
Then the desired inequality would follow trivially by convexity:
\[
\binom{m}{k'}^{-1}\sum_{|R|=k'} \p(Y_R) \leq \binom{m}{k'}^{-1}\sum_{|R|=k'} \sum_{|S|=k} D(R,S) \p(Y_S) = \binom{m}k^{-1} \sum_{|S|=k} \p(Y_S).
\]
The matrix is defined as follows
\begin{align}
D(R,S) &=
\begin{cases}
  \displaystyle \frac{1}{N}\sum_{\substack{E \subset [m]\setminus R \\ |E| = k}}
    \frac{X_{[m]\setminus (E\cup R)}}{X_{[m]\setminus (E\cup S)}},
    & \text{if } S \subset R, \\[3ex]
  \displaystyle\frac{1}{N} \sum_{\substack{F\subset  R \\ |F|=k'-k} }
    \frac{X_F}{X_{[m]\setminus (S\cup (R\setminus F))}},
    & \text{if } S \cap R = \varnothing, \\[3ex]
  0, & \text{otherwise,}
\end{cases}
\end{align}
where $N=\binom{m-k'}{k} \binom{k'}{k}$. We now verify the required properties. Summing over $|S|=k$ we obtain
\begin{align*}
 \sum_{|S|=k} D(R,S)  & = \frac{1}{N} \sum_{\substack{E\subset [m]\setminus R \\ |E| = k \\ S\subset R \\ |S|=k}} \frac{X_{[m]\setminus (E\cup R)}}{X_{[m]\setminus (E\cup S)}}+  \frac{1}{N} \sum_{\substack{F\subset  R \\ |F|=k'-k \\ S\subset [m]\setminus R \\ |S|=k}} \frac{X_F}{X_{[m]\setminus (S\cup (R\setminus F))}} \\
 & =  \frac{1}{N} \sum_{\substack{E\subset [m]\setminus R \\ |E| = k \\ F\subset R \\ |F|=k'-k}} \frac{X_{[m]\setminus (E\cup (R\setminus F))}-X_F}{X_{[m]\setminus (E\cup (R \setminus F))}}+  \frac{1}{N} \sum_{\substack{F\subset  R \\ |F|=k'-k \\ S\subset [m]\setminus R \\ |S|=k}} \frac{X_F}{X_{[m]\setminus (S\cup (R\setminus F))}} \\
 & = \frac{1}{N} \sum_{\substack{E\subset [m]\setminus R \\ |E| = k \\ F\subset R \\ |F|=k'-k}}1 = 1,
\end{align*}
where in the first identity we applied the change of variables $S=R\setminus F$ in the first sum. Summing over $|R|=k'$ we have

\begin{align*}
 \sum_{|R|=k'} D(R,S)  & = \frac{1}{N} \sum_{\substack{E\subset [m]\setminus R \\ |E| = k \\ R\supset S \\ |R|=k'}} \frac{X_{[m]\setminus (E\cup R)}}{X_{[m]\setminus (E\cup S)}}+  \frac{1}{N} \sum_{\substack{F\subset  R \\ |F|=k'-k \\ R\subset [m]\setminus S \\ |R|=k'}} \frac{X_F}{X_{[m]\setminus (S\cup (R\setminus F))}} \\
 &  = \frac{1}{N} \sum_{\substack{E\subset [m]\setminus (S\cup F) \\ |E| = k \\ F\subset [m]\setminus S \\ |F|=k'-k}} \frac{X_{[m]\setminus (E\cup S)}-X_F}{X_{[m]\setminus (E\cup S)}}+  \frac{1}{N} \sum_{\substack{F\subset [m]\setminus S  \\ |F|=k'-k \\ E\subset  [m]\setminus (S\cup F) \\ |E|=k}} \frac{X_F}{X_{[m]\setminus (S\cup E)}}  \\
 & =  \frac{1}{N} \sum_{\substack{E\subset [m]\setminus (S\cup F) \\ |E| = k \\ F\subset [m]\setminus S  \\ |F|=k'-k}}1 = \frac{\binom{m}{k'}}{\binom{m}k},
\end{align*}
where in the first identity we applied the change of variables $R=S \sqcup F$, in the first sum, and $R=E\sqcup F$ in the second sum. Summing $D(R,S)Y_S$ over $|S|=k$, and using the same change of variables as before, we obtain
\begin{align*}
Nk(m-k) \sum_{|S|=k} D(R,S)Y_S  
 & =   \sum_{\substack{E\subset [m]\setminus R \\ |E| = k \\ F\subset R \\ |F|=k'-k}} \frac{X_{[m]\setminus (E\cup R)}X_{R\setminus F} X_{[m]\setminus (R\setminus F)}}{X_{[m]\setminus (E\cup (R \setminus F))}}+   \sum_{\substack{F\subset  R \\ |F|=k'-k \\ E\subset [m]\setminus R \\ |E|=k}} \frac{X_F X_{E} X_{[m]\setminus E}}{X_{[m]\setminus (E\cup (R\setminus F))}}.
\end{align*}
To finish, set $A=[m]\setminus R$. For each fixed pair $(E,F)$ in the sums above, the corresponding summands satisfy
\begin{align*}
&\frac{X_{A\setminus E}X_{R\setminus F}X_{A\cup F}
      +X_FX_EX_{[m]\setminus E}}{X_{(A\setminus E)\cup F}} \\
&\qquad =
\frac{X_{A\setminus E}X_{R\setminus F}
      \bigl(X_{(A\setminus E)\cup F}+X_E\bigr)
      +X_FX_E\bigl(X_{(A\setminus E)\cup F}+X_{R\setminus F}\bigr)}
     {X_{(A\setminus E)\cup F}} \\
&\qquad =
\frac{X_{(A\setminus E)\cup F}\bigl(X_{A\setminus E}X_{R\setminus F} + X_FX_E\bigr) +  \bigl(X_{A\setminus E} + X_F\bigr) X_EX_{R\setminus F}}
     {X_{(A\setminus E)\cup F}} \\
&\qquad 
= X_{R\setminus F}X_A+X_EX_F.
\end{align*}
Consequently,
\begin{align*}
Nk(m-k) \sum_{|S|=k}D(R,S)Y_S
&=\binom{m-k'}{k}X_A\sum_{\substack{F\subset R\\ |F|=k'-k}}X_{R\setminus F}
  +\left(\sum_{\substack{E\subset A\\ |E|=k}}X_E\right)
  \left(\sum_{\substack{F\subset R\\ |F|=k'-k}}X_F\right)\\
&=\binom{m-k'}{k}\binom{k'-1}{k-1}X_AX_R
  +\binom{m-k'-1}{k-1}\binom{k'-1}{k'-k-1}X_AX_R\\
&=N\left(\frac{k}{k'}+
  \frac{k}{m-k'}\frac{k'-k}{k'}\right)X_AX_R\\
&=N\frac{k(m-k)}{k'(m-k')}X_RX_{[m]\setminus R}
 =Nk(m-k)Y_R.
\end{align*}
This proves the last required property of the matrix $D(R,S)$.

\smallskip

\noindent {\bf Necessity.} Take $X_1=2$ and $X_2=\cdots = X_m=1$. It is clear that
\[
\frac{1}{(m-1)!}\sum_{\si\in S_m} \p \bigg[\prod_{j=1}^{n} \bigg(\frac1{a_j}\sum_{i=A_{j-1}+1}^{A_j} X_{\si(i)}\bigg)\bigg]  = \sum_{j=1}^n a_j \p(1+1/a_j).
\]
Since any convex function  $\psi(x) $ for $x\geq0$ can be written as $\psi(x) =x\p(1+1/x)$ for some other convex function $\p(x)$ for $x\geq1$, this completes the proof.
\end{proof}

\section*{Acknowledgments}
F.G. acknowledges support from the following funding agencies: The Office of Naval Research GRANT14201749 (award number N629092412126), The Serrapilheira Institute (Serra-2211-41824), FAPERJ (E-26/200.209/2023 and E-26/210.245/2024) and CNPq (309910/2023-4). J.M. was partially supported by the AMS Stefan Bergman Fellowship and the Simons Foundation Grant $\# 453576$.

\end{document}